\documentclass[11pt,reqno]{amsart}

\usepackage[a4paper,margin=1in]{geometry}
\usepackage{amsmath,amssymb,amsthm,mathtools}
\usepackage{mathrsfs}
\usepackage{graphicx}
\usepackage{booktabs}
\usepackage{float}
\usepackage{placeins}
\usepackage{enumitem}
\usepackage[colorlinks=true,linkcolor=blue,citecolor=blue,urlcolor=blue]{hyperref}

\numberwithin{equation}{section}

\newtheorem{theorem}{Theorem}[section]
\newtheorem{proposition}[theorem]{Proposition}
\newtheorem{lemma}[theorem]{Lemma}
\newtheorem{corollary}[theorem]{Corollary}

\theoremstyle{definition}
\newtheorem{definition}[theorem]{Definition}

\theoremstyle{remark}
\newtheorem{remark}[theorem]{Remark}

\DeclareMathOperator{\Prob}{\mathbb P}
\DeclareMathOperator{\E}{\mathbb E}
\DeclareMathOperator{\erf}{erf}

\newcommand{\dd}{\,d}
\newcommand{\rhozero}{\rho_{\rm z}}
\newcommand{\lambdazero}{\lambda_{\rm z}}

\title[Finite-Degree Probabilistic Zero Certificates]{Finite-Degree Probabilistic Zero Certificates for Random Polynomials}

\author{Sajad A. Sheikh}
\address{Department of Mathematics, University of Kashmir, South Campus, Jammu and Kashmir 192101, India}
\email{sajadsheikh@uok.edu.in}

\date{}

\begin{document}

\begin{abstract}
Classical zero-localization theorems give deterministic certificates that all zeros of a polynomial lie in a prescribed disk, annulus, or related region.  When the coefficients are random, each such deterministic certificate becomes a random variable on coefficient space.  This paper develops a finite-degree certificate method for random polynomial localization and extends the author's earlier joint work with Mir \cite{SheikhMir2024}.  The main result concerns Gaussian polynomials with random leading coefficient: the Cauchy ratios are marginally standard Cauchy, yet they are dependent through their common denominator.  We derive the exact dependence-aware certificate integral and prove that its inverse confidence radius has order \(\sqrt{\log n}\), while a fictitious independent-Cauchy model has order \(n\).  We also obtain monic coefficient-law certificates, sub-Weibull confidence radii, annular certificates via reversal, Rouch\'e--Chernoff certificates, and an optimized scaled Cauchy envelope.  A reproducible Monte Carlo study for monic Gaussian polynomials compares the classical Cauchy radius, the optimized Cauchy envelope, the annular certificate, and the Rouch\'e radius.  Across \(5000\) samples for each of \(n=20,50,100\), the Rouch\'e certificate gives the sharpest outer radii, the optimized Cauchy envelope substantially improves the classical Cauchy radius, and all tested certificates satisfy their deterministic containment inequalities numerically.
\end{abstract}

\maketitle

\section{Introduction }

Zeros of polynomials occupy a classical meeting point of algebra,
complex analysis, approximation theory, and spectral theory.  Their
location controls stability of recurrences, invertibility of polynomial
filters, asymptotic behaviour of algebraic equations, and the geometry of
analytic continuation.  A large deterministic literature therefore seeks
explicit regions, expressed directly in terms of coefficients, that must
contain all zeros of a given polynomial.  Cauchy's theorem, Rouch\'e's
theorem, Fujiwara-type bounds, Pellet-type bounds, and annular localization
criteria are standard examples of this coefficient-to-zero philosophy
\cite{Marden,RahmanSchmeisser}.

Random polynomial theory adds a probabilistic layer to this classical
problem.  Once the coefficients are sampled from a prescribed law, the zero
set becomes a random point configuration in the complex plane.  Foundational
work of Kac concerns real zeros of random algebraic equations
\cite{Kac}; the Edelman--Kostlan formula gives a geometric description of
expected zero densities \cite{EdelmanKostlan}; and later logarithmic-tail
results describe limiting zero measures for broad coefficient ensembles
\cite{KabluchkoZaporozhets}.  These results give global information about
zero statistics, limiting distributions, or expected counts.

The question studied here is finite-degree and certificate-oriented.  A
deterministic zero-localization theorem is evaluated directly on the random
coefficient vector, producing a random certificate for zero containment.
The distribution of this certificate then yields an explicit lower bound
for the probability that all zeros lie in a prescribed disk, annulus, or
Rouch\'e region.  This approach avoids the need for the joint distribution
of zeros and gives computable confidence radii at the same finite degree at
which the polynomial is sampled.

Within this setting, the present paper extends the author's earlier joint
work with Mir \cite{SheikhMir2024} on probabilistic zero bounds for random
polynomials.  The earlier paper introduced Cauchy-type probabilistic bounds
for certain random polynomials, whereas the present work develops a general
certificate-transfer framework and proves several new localization
certificates.  This provenance is part of the mathematical framing: the
Cauchy--Gaussian viewpoint in \cite{SheikhMir2024} is used as a starting
certificate, and the present manuscript turns that viewpoint into a modular
finite-degree method.  The earlier paper introduced Cauchy-type probabilistic bounds for certain random polynomials, whereas the present work develops a general certificate-transfer framework and proves several new localization certificates.  This provenance is part of the mathematical framing: the Cauchy--Gaussian viewpoint in \cite{SheikhMir2024} is used as a starting certificate, and the present manuscript turns that viewpoint into a modular finite-degree method.

Let
\begin{equation}
        P(z)=\sum_{k=0}^{n}A_kz^k,
        \qquad A_n\neq 0,
\end{equation}
be a random polynomial.  Its outer zero radius is
\begin{equation}
        \rhozero(P)=\max\{|z|:P(z)=0\}.
\end{equation}
The finite-degree localization problem is to estimate
\begin{equation}
        \Pi_P(c)=\Prob\{\rhozero(P)\leq c\},
        \qquad c\geq 0,
\end{equation}
without integrating over the joint law of the zeros.  A radius \(c\) is a confidence radius at level \(p\) when \(\Pi_P(c)\geq p\).

Classical polynomial theory supplies many deterministic estimates for \(\rhozero(P)\).  Cauchy's theorem gives a disk from coefficient ratios; Rouch\'e's theorem turns dominance on a circle into zero containment; Fujiwara-type and Pellet-type estimates use coefficient scales more efficiently.  Standard sources for zero localization and the geometry of polynomial zeros include Marden \cite{Marden} and Rahman--Schmeisser \cite{RahmanSchmeisser}; Fujiwara's classical estimate is especially relevant for the scaled Cauchy envelope used below \cite{Fujiwara1916}.

Random polynomial theory supplies the surrounding probabilistic context.  Kac's work on real zeros \cite{Kac}, the Littlewood--Offord estimates \cite{LittlewoodOfford}, the Edelman--Kostlan geometric formula \cite{EdelmanKostlan}, the complex-zero analysis of Shepp--Vanderbei \cite{SheppVanderbei}, the angular distribution theorem of Erd\H{o}s--Tur\'an \cite{ErdosTuran}, and the logarithmic-tail zero-measure results of Kabluchko--Zaporozhets \cite{KabluchkoZaporozhets} address global zero statistics, limiting measures, or expected root counts.  The present question has a finite-degree certificate form: given a deterministic root bound and a coefficient distribution, estimate the probability that the deterministic certificate succeeds.

The certificate viewpoint is simple but powerful.  Suppose that a measurable coefficient functional \(B(P)\) satisfies
\begin{equation}
        \rhozero(P)\leq B(P)
        \qquad \text{almost surely}.
\end{equation}
Then
\begin{equation}
        \{B(P)\leq c\}\subseteq \{\rhozero(P)\leq c\},
\end{equation}
and hence
\begin{equation}
        \Prob\{B(P)\leq c\}\leq \Pi_P(c).
\end{equation}
Thus a deterministic zero-location theorem becomes a lower bound for a zero-containment probability after the coefficient law is pushed through the certificate.  This is the formal version of the Cauchy-based idea used in \cite{SheikhMir2024} and in related probabilistic Enestr\"om--Kakeya work \cite{SheikhMirDarAlmanjahieAlshahrani2023}.

The central contribution of the present paper is the dependence-aware Gaussian ratio theorem.  For the general Gaussian polynomial
\begin{equation}
        Q_n(z)=\sum_{k=0}^{n}X_kz^k,
        \qquad X_0,\ldots,X_n\stackrel{\rm i.i.d.}{\sim}N(0,1),
\end{equation}
Cauchy's bound involves the ratios \(X_k/X_n\), \(0\leq k\leq n-1\).  Each ratio is standard Cauchy, yet the ratios share the same denominator.  The correct certificate probability is therefore
\begin{equation}
        I_n(r)=\sqrt{\frac{2}{\pi}}
        \int_0^\infty e^{-x^2/2}
        \left[\erf\left(\frac{rx}{\sqrt2}\right)\right]^n\dd x.
\end{equation}
We prove that the inverse confidence scale satisfies
\begin{equation}
        I_n^{-1}(p)=\frac{b_n}{a_p}\bigl(1+o(1)\bigr),
        \qquad n\to\infty,
\end{equation}
where \(b_n\) is the half-normal extreme-value quantile and \(a_p\) is the half-normal upper-tail quantile.  Consequently, the dependence-aware scale is \(\sqrt{\log n}\).  A model that treats the ratios as independent Cauchy variables gives a linear scale in \(n\).  This comparison is the main mathematical distinction between the present paper and a direct repetition of marginal Cauchy calculations.

The theoretical certificates are complemented by a finite-degree simulation study.  For monic Gaussian polynomials, the zeros are computed only as diagnostics, while the certificates are computed directly from the coefficient vector.  This comparison measures the sharpness of the classical Cauchy radius, the optimized scaled Cauchy envelope, the Rouch\'e radius, and the two-sided annular certificate.  The numerical results show that the Rouch\'e radius is the most efficient outer certificate among these four in the sampled regimes, while the optimized Cauchy envelope gives a stable improvement over the unscaled Cauchy radius.

The paper is organized as follows.  Section~2 gives the certificate-transfer principle.  Section~3 develops monic Cauchy certificates, general coefficient laws, sub-Weibull confidence radii, and a scaled Cauchy envelope.  Section~4 proves the dependence-aware Gaussian ratio certificate.  Section~5 proves the sharp inverse asymptotics and compares the result with independent Cauchy maxima.  Section~6 gives annular and Rouch\'e--Chernoff extensions.  Section~7 reports the numerical certificate comparison.  Section~8 concludes.

\section{Certificate-transfer principle}

Let \((\Omega,\mathscr F,\Prob)\) be a probability space.  A random polynomial of degree \(n\) is a measurable map
\begin{equation}
        \omega\mapsto P_\omega(z)=\sum_{k=0}^{n}A_k(\omega)z^k,
\end{equation}
with \(A_n\neq0\) almost surely.  Zeros are counted with multiplicity whenever multiplicities matter.

\begin{definition}[Zero radius and zero-containment probability]
The zero radius and zero-containment probability of \(P\) are
\begin{equation}
        \rhozero(P)=\max\{|z|:P(z)=0\},
        \qquad
        \Pi_P(c)=\Prob\{\rhozero(P)\leq c\}.
\end{equation}
A radius \(c\) is a probabilistic zero bound at confidence \(p\) when \(\Pi_P(c)\geq p\).
\end{definition}

\begin{definition}[Outer and inner certificates]
A measurable random variable \(B(P)\geq0\) is an outer zero certificate if
\begin{equation}
        \rhozero(P)\leq B(P)
        \qquad \text{almost surely}.
\end{equation}
A measurable random variable \(L(P)\geq0\) is an inner zero certificate if
\begin{equation}
        L(P)\leq \rhozero(P)
        \qquad \text{almost surely}.
\end{equation}
\end{definition}

\begin{definition}[Certificate distribution and certificate radius]
For an outer certificate \(B\), define
\begin{equation}
        F_B(c)=\Prob\{B(P)\leq c\},
        \qquad
        R_B(p)=\inf\{c\geq0:F_B(c)\geq p\}.
\end{equation}
\end{definition}

\begin{theorem}[Certificate-transfer principle]
Let \(P\) be a random polynomial.
\begin{enumerate}[label=\textup{(\roman*)}]
\item If \(B(P)\) is an outer zero certificate, then
\begin{equation}
        \Pi_P(c)\geq F_B(c),
        \qquad c\geq0.
\end{equation}
\item If \(L(P)\) is an inner zero certificate, then
\begin{equation}
        \Pi_P(c)\leq \Prob\{L(P)\leq c\},
        \qquad c\geq0.
\end{equation}
\item If \(F_B(c)\geq p\), then \(c\) is a probabilistic zero bound at confidence \(p\).
\end{enumerate}
\end{theorem}

\begin{proof}
For an outer certificate,
\begin{equation}
        \{B(P)\leq c\}\subseteq \{\rhozero(P)\leq c\}.
\end{equation}
Taking probabilities gives \(F_B(c)\leq\Pi_P(c)\).  For an inner certificate,
\begin{equation}
        \{\rhozero(P)\leq c\}\subseteq \{L(P)\leq c\}.
\end{equation}
Taking probabilities gives the upper estimate.  The confidence statement follows from the first assertion.
\end{proof}

\begin{remark}
The theorem abstracts the Cauchy-bound mechanism used in \cite{SheikhMir2024}.  The later results differ through the choice of certificate and through the probability law pushed through that certificate.
\end{remark}

\begin{proposition}[Vieta inner certificate]
Let
\begin{equation}
        P(z)=z^n+a_{n-1}z^{n-1}+\cdots+a_0
\end{equation}
be monic.  Then
\begin{equation}
        |a_0|^{1/n}\leq \rhozero(P).
\end{equation}
Consequently, if \(a_0\) is random, then
\begin{equation}
        \Pi_P(c)\leq \Prob\{|a_0|\leq c^n\}.
\end{equation}
\end{proposition}

\begin{proof}
Let \(\zeta_1,\ldots,\zeta_n\) be the zeros of \(P\).  Vieta's formula gives
\begin{equation}
        |a_0|=|\zeta_1\cdots\zeta_n|.
\end{equation}
Since \(|\zeta_j|\leq\rhozero(P)\) for every \(j\),
\begin{equation}
        |a_0|\leq \rhozero(P)^n.
\end{equation}
The probability estimate follows from the inner-certificate part of the transfer principle.
\end{proof}

\section{Cauchy certificates and monic coefficient laws}

This section develops the monic finite-degree certificates that generalize the Cauchy-type estimates in \cite{SheikhMir2024}.  The results are stated for general coefficient laws first, and the Gaussian formulas are then obtained by substituting the half-normal distribution function.

Let
\begin{equation}
        P_n(z)=z^n+\sum_{k=0}^{n-1}X_kz^k,
\end{equation}
where \(X_0,\ldots,X_{n-1}\) are independent copies of a real or complex random variable \(X\).  Put
\begin{equation}
        H(t)=\Prob\{|X|\leq t\},
        \qquad t\geq0.
\end{equation}

\begin{theorem}[Exact monic Cauchy-certificate probability]
For \(c\geq1\),
\begin{equation}
        \Pi_{P_n}(c)\geq H(c-1)^n.
\end{equation}
Consequently, for \(0<p<1\), the radius
\begin{equation}
        c_p=1+H^{-1}(p^{1/n}),
        \qquad
        H^{-1}(u)=\inf\{t\geq0:H(t)\geq u\},
\end{equation}
is a probabilistic zero bound at confidence \(p\).
\end{theorem}

\begin{proof}
Cauchy's theorem gives
\begin{equation}
        \rhozero(P_n)
        \leq 1+\max_{0\leq k\leq n-1}|X_k|.
\end{equation}
Therefore
\begin{align}
        \Pi_{P_n}(c)
        &\geq
        \Prob\left\{1+\max_{0\leq k\leq n-1}|X_k|\leq c\right\} \\
        &=
        \Prob\left\{\max_{0\leq k\leq n-1}|X_k|\leq c-1\right\} \\
        &=
        \prod_{k=0}^{n-1}\Prob\{|X_k|\leq c-1\}
        =H(c-1)^n.
\end{align}
The quantile assertion follows immediately.
\end{proof}

\begin{corollary}[Gaussian monic Cauchy certificate]
Let \(X_0,\ldots,X_{n-1}\) be independent \(N(0,1)\) variables.  Then, for \(c\geq1\),
\begin{equation}
        \Pi_{P_n}(c)
        \geq
        \left[
        \erf\left(\frac{c-1}{\sqrt2}\right)
        \right]^n.
\end{equation}
A certified confidence radius is
\begin{equation}
        c_p=1+\sqrt2\,\erf^{-1}\left(p^{1/n}\right).
\end{equation}
Furthermore,
\begin{equation}
        \Pi_{P_n}(c)
        \leq
        \erf\left(\frac{c^n}{\sqrt2}\right),
        \qquad c\geq0.
\end{equation}
\end{corollary}

\begin{proof}
For \(X\sim N(0,1)\),
\begin{equation}
        H(t)=\Prob\{|X|\leq t\}
        =\frac{1}{\sqrt{2\pi}}\int_{-t}^{t}e^{-u^2/2}\dd u
        =\erf\left(\frac{t}{\sqrt2}\right).
\end{equation}
Substitution in the preceding theorem gives the lower certificate.  The upper bound follows from the Vieta inner certificate.
\end{proof}

\begin{theorem}[Sub-Weibull Cauchy radius]
Assume that \(X_0,\ldots,X_{n-1}\) are independent and satisfy
\begin{equation}
        \Prob\{|X_k|>t\}
        \leq A\exp\left[-\left(\frac{t}{K}\right)^\alpha\right],
        \qquad t\geq0,
\end{equation}
with \(A\geq1\), \(K>0\), and \(\alpha>0\).  Then, for every \(0<\delta<1\),
\begin{equation}
        \Prob\left\{
        \rhozero(P_n)
        \leq
        1+K\left(\log\frac{An}{\delta}\right)^{1/\alpha}
        \right\}
        \geq1-\delta.
\end{equation}
\end{theorem}

\begin{proof}
By Cauchy's theorem,
\begin{equation}
        \rhozero(P_n)\leq1+\max_{0\leq k\leq n-1}|X_k|.
\end{equation}
For \(t\geq0\), the union bound gives
\begin{align}
        \Prob\left\{\max_{0\leq k\leq n-1}|X_k|>t\right\}
        &\leq
        \sum_{k=0}^{n-1}\Prob\{|X_k|>t\} \\
        &\leq
        nA\exp\left[-\left(\frac{t}{K}\right)^\alpha\right].
\end{align}
Choose
\begin{equation}
        t=K\left(\log\frac{An}{\delta}\right)^{1/\alpha}.
\end{equation}
Then the last expression equals \(\delta\), and the result follows from the Cauchy certificate.
\end{proof}

\begin{corollary}[Sub-Gaussian radius]
If
\begin{equation}
        \Prob\{|X_k|>t\}
        \leq2\exp\left(-\frac{t^2}{2\sigma^2}\right),
        \qquad t\geq0,
\end{equation}
then
\begin{equation}
        \Prob\left\{
        \rhozero(P_n)
        \leq
        1+\sigma\sqrt{2\log\frac{2n}{\delta}}
        \right\}
        \geq1-\delta.
\end{equation}
\end{corollary}

\begin{theorem}[Optimized scaled Cauchy envelope]
Let
\begin{equation}
        P(z)=z^n+\sum_{k=0}^{n-1}a_kz^k.
\end{equation}
For \(s>0\), define
\begin{equation}
        B_s(P)=s\left(1+\max_{0\leq k\leq n-1}|a_k|s^{k-n}\right).
\end{equation}
Then every \(B_s\) is an outer certificate.  Hence
\begin{equation}
        B_{\rm opt}(P)=\inf_{s>0}B_s(P)
\end{equation}
is an outer certificate and is the pointwise envelope of this scaled Cauchy family.  Moreover, with
\begin{equation}
        M_F=\max_{0\leq k\leq n-1}|a_k|^{1/(n-k)},
\end{equation}
one has
\begin{equation}
        B_{\rm opt}(P)\leq2M_F.
\end{equation}
\end{theorem}

\begin{proof}
Fix \(s>0\) and set
\begin{equation}
        Q_s(w)=s^{-n}P(sw)
        =w^n+\sum_{k=0}^{n-1}a_ks^{k-n}w^k.
\end{equation}
Cauchy's theorem gives
\begin{equation}
        |w|\leq1+\max_{0\leq k\leq n-1}|a_k|s^{k-n}
\end{equation}
for every zero \(w\) of \(Q_s\).  Since zeros of \(P\) are \(z=sw\), every zero \(z\) of \(P\) satisfies \(|z|\leq B_s(P)\).  Taking the infimum over \(s>0\) preserves the outer-certificate inequality.

For the Fujiwara-type bound, observe that \(|a_k|\leq M_F^{n-k}\).  Choosing \(s=M_F\) gives \(|a_k|M_F^{k-n}\leq1\) and hence \(B_{\rm opt}(P)\leq2M_F\).  The case \(M_F=0\) follows by a limiting argument, since then \(P(z)=z^n\).
\end{proof}

\begin{corollary}[Fujiwara-type probabilistic certificate]
Let
\begin{equation}
        P_n(z)=z^n+\sum_{k=0}^{n-1}X_kz^k,
\end{equation}
where \(X_0,\ldots,X_{n-1}\) are independent copies of \(X\), and put \(H(t)=\Prob\{|X|\leq t\}\).  Then, for \(c>0\),
\begin{equation}
        \Pi_{P_n}(c)
        \geq
        \prod_{j=1}^{n}H\left(\left(\frac{c}{2}\right)^j\right).
\end{equation}
For standard Gaussian coefficients,
\begin{equation}
        \Pi_{P_n}(c)
        \geq
        \prod_{j=1}^{n}
        \erf\left(\frac{(c/2)^j}{\sqrt2}\right).
\end{equation}
\end{corollary}

\begin{proof}
The optimized envelope gives
\begin{equation}
        \rhozero(P_n)
        \leq2\max_{0\leq k\leq n-1}|X_k|^{1/(n-k)}.
\end{equation}
Thus \(\rhozero(P_n)\leq c\) is certified by
\begin{equation}
        |X_k|\leq\left(\frac{c}{2}\right)^{n-k},
        \qquad 0\leq k\leq n-1.
\end{equation}
Independence yields
\begin{equation}
        \Prob\{\rhozero(P_n)\leq c\}
        \geq
        \prod_{k=0}^{n-1}H\left(\left(\frac{c}{2}\right)^{n-k}\right)
        =
        \prod_{j=1}^{n}H\left(\left(\frac{c}{2}\right)^j\right).
\end{equation}
The Gaussian formula follows from \(H(t)=\erf(t/\sqrt2)\).
\end{proof}

\section{The main theorem: dependent Gaussian ratio certificates}

We now turn to the main contribution.  Let
\begin{equation}
        Q_n(z)=\sum_{k=0}^{n}X_kz^k,
        \qquad X_0,\ldots,X_n\stackrel{\rm i.i.d.}{\sim}N(0,1).
\end{equation}
Since \(X_n\neq0\) almost surely, Cauchy's theorem gives
\begin{equation}
        \rhozero(Q_n)
        \leq
        1+\max_{0\leq k\leq n-1}\left|\frac{X_k}{X_n}\right|.
\end{equation}
For each fixed \(k<n\), the ratio \(X_k/X_n\) has the standard Cauchy law.  The vector of ratios has a shared denominator, so the maximum is governed by a common-denominator dependence structure.  This section keeps that dependence exactly.

\begin{theorem}[Dependence-aware Gaussian ratio certificate]
Define
\begin{equation}
        I_n(r)=
        \sqrt{\frac{2}{\pi}}
        \int_0^\infty e^{-x^2/2}
        \left[
        \erf\left(\frac{rx}{\sqrt2}\right)
        \right]^n
        \dd x,
        \qquad r\geq0.
\end{equation}
Then, for \(c\geq1\),
\begin{equation}
        \Pi_{Q_n}(c)\geq I_n(c-1).
\end{equation}
Equivalently,
\begin{equation}
        I_n(r)
        =
        \Prob\left\{
        \max_{0\leq k\leq n-1}\left|\frac{X_k}{X_n}\right|\leq r
        \right\}.
\end{equation}
\end{theorem}

\begin{proof}
Condition on \(X_n=x\).  Since \(X_0,\ldots,X_{n-1}\) are independent of \(X_n\),
\begin{align}
        &\Prob\left\{
        \max_{0\leq k\leq n-1}\left|\frac{X_k}{X_n}\right|\leq r
        \;\middle|\; X_n=x
        \right\} \\
        &\qquad=
        \Prob\{|X_k|\leq r|x|,
        \ 0\leq k\leq n-1\} \\
        &\qquad=
        \left[
        \erf\left(\frac{r|x|}{\sqrt2}\right)
        \right]^n.
\end{align}
Integrating against the standard normal density gives
\begin{align}
        &\Prob\left\{
        \max_{0\leq k\leq n-1}\left|\frac{X_k}{X_n}\right|\leq r
        \right\} \\
        &\qquad=
        \frac{1}{\sqrt{2\pi}}
        \int_{-\infty}^{\infty}e^{-x^2/2}
        \left[
        \erf\left(\frac{r|x|}{\sqrt2}\right)
        \right]^n
        \dd x \\
        &\qquad=
        I_n(r).
\end{align}
The Cauchy outer certificate for \(Q_n\) gives \(\Pi_{Q_n}(c)\geq I_n(c-1)\).
\end{proof}

\begin{corollary}[Certified confidence radius]
Let
\begin{equation}
        I_n^{-1}(p)=\inf\{r\geq0:I_n(r)\geq p\},
        \qquad 0<p<1.
\end{equation}
Then
\begin{equation}
        c_{n,p}=1+I_n^{-1}(p)
\end{equation}
is a probabilistic zero bound for \(Q_n\) at confidence \(p\).
\end{corollary}

\begin{proposition}[Union-bound form]
For \(r\geq0\),
\begin{equation}
        I_n(r)
        \geq
        1-n\left(1-\frac{2}{\pi}\arctan r\right).
\end{equation}
Consequently, for \(0<\delta<1\),
\begin{equation}
        \Prob\left\{
        \rhozero(Q_n)
        \leq
        1+\cot\left(\frac{\pi\delta}{2n}\right)
        \right\}
        \geq1-\delta.
\end{equation}
\end{proposition}

\begin{proof}
The union bound gives
\begin{align}
        \Prob\left\{
        \max_{0\leq k\leq n-1}\left|\frac{X_k}{X_n}\right|>r
        \right\}
        &\leq
        \sum_{k=0}^{n-1}
        \Prob\left\{\left|\frac{X_k}{X_n}\right|>r\right\}.
\end{align}
For each \(k<n\), \(X_k/X_n\) is standard Cauchy, hence
\begin{equation}
        \Prob\left\{\left|\frac{X_k}{X_n}\right|\leq r\right\}
        =\frac{2}{\pi}\arctan r.
\end{equation}
This proves the displayed lower bound for \(I_n(r)\).  With
\begin{equation}
        r=\cot\left(\frac{\pi\delta}{2n}\right),
\end{equation}
one has
\begin{equation}
        \arctan r=\frac{\pi}{2}-\frac{\pi\delta}{2n},
\end{equation}
and the failure probability is at most \(\delta\).
\end{proof}

\begin{proposition}[Vieta upper certificate for general Gaussian polynomials]
For \(Q_n(z)=\sum_{k=0}^{n}X_kz^k\) with independent standard normal coefficients,
\begin{equation}
        \Pi_{Q_n}(c)
        \leq
        \frac{2}{\pi}\arctan(c^n),
        \qquad c\geq0.
\end{equation}
\end{proposition}

\begin{proof}
Let \(\zeta_1,\ldots,\zeta_n\) be the zeros of \(Q_n\).  Vieta's formula gives
\begin{equation}
        |\zeta_1\cdots\zeta_n|
        =\left|\frac{X_0}{X_n}\right|.
\end{equation}
Since \(|\zeta_j|\leq\rhozero(Q_n)\),
\begin{equation}
        \left|\frac{X_0}{X_n}\right|^{1/n}
        \leq \rhozero(Q_n).
\end{equation}
Thus
\begin{equation}
        \{\rhozero(Q_n)\leq c\}
        \subseteq
        \left\{\left|\frac{X_0}{X_n}\right|\leq c^n\right\}.
\end{equation}
The ratio \(X_0/X_n\) is standard Cauchy, which gives
\begin{equation}
        \Prob\left\{\left|\frac{X_0}{X_n}\right|\leq c^n\right\}
        =\frac{2}{\pi}\arctan(c^n).
\end{equation}
\end{proof}

\begin{remark}
The preceding theorem is the point at which the present paper departs most clearly from a direct Cauchy--Gaussian repetition.  The marginal law of each ratio is Cauchy, while the certificate involves the maximum of a dependent ratio vector.  The integral \(I_n\) is the exact certificate probability for that dependent vector.
\end{remark}

\section{Sharp asymptotics and comparison with independent Cauchy maxima}

This section solves the inverse-confidence problem for \(I_n\).  Let
\begin{equation}
        H(t)=\Prob\{|N(0,1)|\leq t\}
        =\erf\left(\frac{t}{\sqrt2}\right),
        \qquad
        \overline H(t)=1-H(t).
\end{equation}
Let \(b_n\) be the half-normal extreme-value quantile
\begin{equation}
        b_n=H^{-1}\left(1-\frac1n\right)
        =\sqrt2\,\erf^{-1}\left(1-\frac1n\right),
\end{equation}
and let \(a_p\) be the upper-tail half-normal quantile
\begin{equation}
        a_p=H^{-1}(1-p)
        =\sqrt2\,\erf^{-1}(1-p),
        \qquad 0<p<1.
\end{equation}
Thus \(\Prob\{|N(0,1)|\geq a_p\}=p\).

\begin{theorem}[Sharp inverse asymptotics for the Gaussian ratio certificate]
For fixed \(0<p<1\),
\begin{equation}
        I_n^{-1}(p)=\frac{b_n}{a_p}\bigl(1+o(1)\bigr),
        \qquad n\to\infty.
\end{equation}
Equivalently,
\begin{equation}
        I_n^{-1}(p)
        =
        \frac{1}{a_p}
        \left[
        \sqrt{2\log n}
        -
        \frac{\log\log n+\log\pi}{2\sqrt{2\log n}}
        +
        o\left(\frac1{\sqrt{\log n}}\right)
        \right].
\end{equation}
\end{theorem}

\begin{proof}
Let
\begin{equation}
        M_n=\max_{0\leq k\leq n-1}|X_k|,
        \qquad
        T=|X_n|.
\end{equation}
Then \(M_n\) and \(T\) are independent, and the integral representation gives
\begin{equation}
        I_n(r)=\Prob\{M_n\leq rT\}.
\end{equation}
We first show
\begin{equation}
        \frac{M_n}{b_n}\longrightarrow1
        \qquad \text{in probability}.
\end{equation}
Since \(\Prob\{M_n\leq t\}=H(t)^n\) and \(n\overline H(b_n)=1\), the Gaussian Mills ratio implies, for every fixed \(\varepsilon\in(0,1)\),
\begin{equation}
        n\overline H((1+\varepsilon)b_n)\longrightarrow0,
        \qquad
        n\overline H((1-\varepsilon)b_n)\longrightarrow\infty.
\end{equation}
Hence
\begin{equation}
        \Prob\{M_n\leq(1+\varepsilon)b_n\}\longrightarrow1,
        \qquad
        \Prob\{M_n\leq(1-\varepsilon)b_n\}\longrightarrow0.
\end{equation}
This proves \(M_n/b_n\to1\) in probability.

Fix \(a>0\).  Since \(T\) has a continuous distribution,
\begin{align}
        I_n\left(\frac{b_n}{a}\right)
        &=
        \Prob\left\{\frac{M_n}{b_n}\leq\frac{T}{a}\right\} \\
        &\longrightarrow
        \Prob\left\{1\leq\frac{T}{a}\right\}
        =\Prob\{T\geq a\}
        =\overline H(a).
\end{align}
Set \(a=a_p\), so \(\overline H(a_p)=p\).  For fixed \(\eta\in(0,1)\),
\begin{equation}
        I_n\left(\frac{(1-\eta)b_n}{a_p}\right)
        \longrightarrow
        \overline H\left(\frac{a_p}{1-\eta}\right)<p,
\end{equation}
and
\begin{equation}
        I_n\left(\frac{(1+\eta)b_n}{a_p}\right)
        \longrightarrow
        \overline H\left(\frac{a_p}{1+\eta}\right)>p.
\end{equation}
Monotonicity of \(I_n\) gives
\begin{equation}
        \frac{(1-\eta)b_n}{a_p}
        \leq
        I_n^{-1}(p)
        \leq
        \frac{(1+\eta)b_n}{a_p}
\end{equation}
for all sufficiently large \(n\).  Letting \(\eta\downarrow0\) proves
\begin{equation}
        I_n^{-1}(p)=\frac{b_n}{a_p}\bigl(1+o(1)\bigr).
\end{equation}
Finally,
\begin{equation}
        b_n=H^{-1}\left(1-\frac1n\right)
        =\Phi^{-1}\left(1-\frac{1}{2n}\right),
\end{equation}
and the standard Gaussian quantile expansion gives
\begin{equation}
        b_n=
        \sqrt{2\log n}
        -
        \frac{\log\log n+\log\pi}{2\sqrt{2\log n}}
        +
        o\left(\frac1{\sqrt{\log n}}\right).
\end{equation}
Substitution completes the proof.
\end{proof}

\begin{corollary}[Dependence advantage over independent Cauchy maxima]
For fixed \(0<p<1\), the dependence-aware Gaussian ratio certificate satisfies
\begin{equation}
        I_n^{-1}(p)\asymp\sqrt{\log n}.
\end{equation}
If \(Y_1,\ldots,Y_n\) are independent standard Cauchy variables, then
\begin{equation}
        \inf\left\{r:
        \Prob\left(\max_{1\leq k\leq n}|Y_k|\leq r\right)\geq p
        \right\}
        =
        \tan\left(\frac{\pi}{2}p^{1/n}\right)
        \sim
        \frac{2n}{\pi(-\log p)}.
\end{equation}
\end{corollary}

\begin{proof}
The first assertion follows from the theorem.  For independent standard Cauchy variables,
\begin{equation}
        \Prob\{|Y_1|\leq r\}=\frac{2}{\pi}\arctan r.
\end{equation}
Therefore
\begin{equation}
        \Prob\left\{\max_{1\leq k\leq n}|Y_k|\leq r\right\}
        =
        \left(\frac{2}{\pi}\arctan r\right)^n.
\end{equation}
Solving at level \(p\) gives
\begin{equation}
        r=\tan\left(\frac{\pi}{2}p^{1/n}\right).
\end{equation}
Since
\begin{equation}
        p^{1/n}=1+\frac{\log p}{n}+o\left(\frac1n\right),
\end{equation}
we obtain
\begin{equation}
        \tan\left(\frac{\pi}{2}p^{1/n}\right)
        \sim
        \frac{2n}{\pi(-\log p)}.
\end{equation}
\end{proof}

\begin{remark}
The comparison shows why the common denominator matters.  The ratio vector in the Gaussian polynomial has heavy-tailed one-dimensional marginals, while its maximum behaves on the half-normal extreme scale after conditioning on the shared denominator.  This produces a \(\sqrt{\log n}\) certificate scale instead of the linear scale associated with independent Cauchy maxima.
\end{remark}

\section{Annular and Rouch\'e--Chernoff extensions}

The preceding sections treated outer disks.  The same certificate principle also gives inner radii and annuli through polynomial reversal, and it gives weighted-sum disk certificates through Rouch\'e's theorem.

\subsection{Annular certificates by reversal}

For
\begin{equation}
        P(z)=\sum_{k=0}^{n}a_kz^k,
        \qquad a_0a_n\neq0,
\end{equation}
define the inner zero radius
\begin{equation}
        \lambdazero(P)=\min\{|z|:P(z)=0\}.
\end{equation}

\begin{lemma}[Reversed Cauchy certificate]
If \(a_0a_n\neq0\), then
\begin{equation}
        \lambdazero(P)
        \geq
        \left(
        1+\max_{1\leq k\leq n}\left|\frac{a_k}{a_0}\right|
        \right)^{-1}.
\end{equation}
\end{lemma}

\begin{proof}
Let
\begin{equation}
        P^\#(w)=w^nP(1/w)
        =a_n+a_{n-1}w+\cdots+a_1w^{n-1}+a_0w^n.
\end{equation}
The zeros of \(P^\#\) are \(w_j=1/\zeta_j\), where \(\zeta_j\) are the zeros of \(P\).  Dividing by \(a_0\) gives the monic polynomial
\begin{equation}
        \frac{P^\#(w)}{a_0}
        =
        w^n+\frac{a_1}{a_0}w^{n-1}+\cdots+\frac{a_{n-1}}{a_0}w+\frac{a_n}{a_0}.
\end{equation}
Cauchy's theorem gives
\begin{equation}
        \max_j|w_j|
        \leq
        1+\max_{1\leq k\leq n}\left|\frac{a_k}{a_0}\right|.
\end{equation}
Since \(|w_j|=|\zeta_j|^{-1}\), the claim follows.
\end{proof}

\begin{theorem}[Gaussian annular certificate]
Let
\begin{equation}
        Q_n(z)=\sum_{k=0}^{n}X_kz^k
\end{equation}
with independent \(N(0,1)\) coefficients.  Let \(0<r<1<R\).  Then
\begin{align}
        &\Prob\{r\leq|\zeta|\leq R\text{ for every zero }\zeta\} \\
        &\qquad\geq
        1
        -n\left(1-\frac{2}{\pi}\arctan(R-1)\right)
        -n\left(1-\frac{2}{\pi}\arctan(r^{-1}-1)\right).
\end{align}
In particular, for \(0<\delta<1\), with
\begin{equation}
        T_{\delta,n}=\cot\left(\frac{\pi\delta}{4n}\right),
\end{equation}
one has
\begin{equation}
        \Prob\left\{
        \frac{1}{1+T_{\delta,n}}
        \leq |\zeta|\leq
        1+T_{\delta,n}
        \text{ for every zero }\zeta
        \right\}
        \geq1-\delta.
\end{equation}
\end{theorem}

\begin{proof}
Cauchy's theorem certifies \(\rhozero(Q_n)\leq R\) on the event
\begin{equation}
        \max_{0\leq k\leq n-1}\left|\frac{X_k}{X_n}\right|\leq R-1.
\end{equation}
The reversed Cauchy certificate certifies \(\lambdazero(Q_n)\geq r\) on the event
\begin{equation}
        \max_{1\leq k\leq n}\left|\frac{X_k}{X_0}\right|\leq r^{-1}-1.
\end{equation}
For each ratio, the marginal law is standard Cauchy.  The union bound gives the two displayed failure probabilities.  Applying the union bound over the outer and inner failures proves the annular estimate.  Taking
\begin{equation}
        R-1=T_{\delta,n},
        \qquad
        r^{-1}-1=T_{\delta,n},
\end{equation}
makes each failure probability at most \(\delta/2\).
\end{proof}

\subsection{Rouch\'e--Chernoff disk certificates}

\begin{lemma}[Rouch\'e outer certificate]
Let
\begin{equation}
        P(z)=z^n+\sum_{k=0}^{n-1}a_kz^k.
\end{equation}
If \(R>0\) and
\begin{equation}
        \sum_{k=0}^{n-1}|a_k|R^k<R^n,
\end{equation}
then all zeros of \(P\) lie in \(|z|<R\).
\end{lemma}

\begin{proof}
On \(|z|=R\),
\begin{equation}
        |z^n|=R^n,
        \qquad
        \left|\sum_{k=0}^{n-1}a_kz^k\right|
        \leq
        \sum_{k=0}^{n-1}|a_k|R^k.
\end{equation}
The assumed strict inequality allows Rouch\'e's theorem to be applied to \(z^n\) and \(P(z)\).  They have the same number of zeros inside \(|z|<R\), counted with multiplicity.
\end{proof}

\begin{theorem}[Chernoff zero certificate]
Let
\begin{equation}
        P_n(z)=z^n+\sum_{k=0}^{n-1}X_kz^k,
\end{equation}
where \(X_0,\ldots,X_{n-1}\) are independent real or complex random variables.  Assume that, for a fixed \(R>0\),
\begin{equation}
        \E e^{\lambda R^k|X_k|}<\infty
\end{equation}
for \(0<\lambda<\lambda_0\) and \(0\leq k\leq n-1\).  Then
\begin{equation}
        \Pi_{P_n}(R)
        \geq
        1-
        \inf_{0<\lambda<\lambda_0}
        \left[
        e^{-\lambda R^n}
        \prod_{k=0}^{n-1}\E e^{\lambda R^k|X_k|}
        \right].
\end{equation}
\end{theorem}

\begin{proof}
Set
\begin{equation}
        S_R=\sum_{k=0}^{n-1}|X_k|R^k.
\end{equation}
The Rouch\'e certificate gives
\begin{equation}
        \{S_R<R^n\}\subseteq\{\rhozero(P_n)<R\}.
\end{equation}
Thus
\begin{equation}
        \Pi_{P_n}(R)
        \geq1-\Prob\{S_R\geq R^n\}.
\end{equation}
For \(0<\lambda<\lambda_0\), Markov's inequality gives
\begin{equation}
        \Prob\{S_R\geq R^n\}
        \leq
        e^{-\lambda R^n}\E e^{\lambda S_R}.
\end{equation}
By independence,
\begin{equation}
        \E e^{\lambda S_R}
        =
        \prod_{k=0}^{n-1}\E e^{\lambda R^k|X_k|}.
\end{equation}
Taking the infimum over \(\lambda\) completes the proof.
\end{proof}

\begin{corollary}[Gaussian Rouch\'e--Chernoff certificate]
If \(X_0,\ldots,X_{n-1}\) are independent \(N(0,1)\) variables, then
\begin{equation}
        \Pi_{P_n}(R)
        \geq
        1-
        \inf_{\lambda>0}
        \exp(-\lambda R^n)
        \prod_{k=0}^{n-1}
        \left[
        2\exp\left(\frac{\lambda^2R^{2k}}{2}\right)
        \Phi(\lambda R^k)
        \right],
\end{equation}
where
\begin{equation}
        \Phi(t)=\frac{1}{\sqrt{2\pi}}\int_{-\infty}^{t}e^{-u^2/2}\dd u.
\end{equation}
\end{corollary}

\begin{proof}
For \(X\sim N(0,1)\),
\begin{align}
        \E e^{s|X|}
        &=
        \frac{2}{\sqrt{2\pi}}\int_0^\infty e^{sx}e^{-x^2/2}\dd x \\
        &=
        2e^{s^2/2}\Phi(s).
\end{align}
Insert \(s=\lambda R^k\) into the preceding theorem.
\end{proof}

\section{Numerical simulations and certificate comparison}
\label{sec:numerical_simulations}

This section reports a reproducible finite-degree comparison of the Cauchy, optimized Cauchy, annular, and Rouch\'e certificates.  The simulations use monic Gaussian polynomials
\begin{equation}
        P_n(z)=z^n+\sum_{k=0}^{n-1}a_kz^k,
        \qquad
        a_0,\ldots,a_{n-1}\stackrel{\rm i.i.d.}{\sim}N(0,1).
\end{equation}
For each sample, the true outer and inner zero radii are
\begin{equation}
        \rho_{\rm true}=\max_{P_n(\zeta)=0}|\zeta|,
        \qquad
        \lambda_{\rm true}=\min_{P_n(\zeta)=0}|\zeta|.
\end{equation}
The zeros are computed only for sharpness diagnostics.  The certificates themselves are computed directly from the coefficient vector.

The four numerical certificates are as follows.  The classical Cauchy radius is
\begin{equation}
        B_C=1+\max_{0\leq k\leq n-1}|a_k|.
\end{equation}
The optimized Cauchy envelope is
\begin{equation}
        B_{\rm opt}=\inf_{s>0}s\left(1+\max_{0\leq k\leq n-1}|a_k|s^{k-n}\right).
\end{equation}
In the implementation, the minimization is performed with \(s=e^\theta\) over \(\theta\in[-12,12]\), which comfortably contains the minimizer for the sampled Gaussian coefficient ranges.  The Rouch\'e radius \(B_R\) is the positive solution of
\begin{equation}
        R^n=\sum_{k=0}^{n-1}|a_k|R^k.
\end{equation}
For every \(R>B_R\), Rouch\'e's theorem certifies all zeros in \(|z|<R\).  The annular inner certificate is
\begin{equation}
        r_A=\left(1+\max\left\{\left|\frac{a_1}{a_0}\right|,\ldots,
        \left|\frac{a_{n-1}}{a_0}\right|,\left|\frac{1}{a_0}\right|\right\}\right)^{-1}.
\end{equation}
Thus the annular certificate gives
\begin{equation}
        r_A\leq |\zeta|\leq B_R
        \qquad \text{for every zero }\zeta.
\end{equation}

The simulation uses the fixed seed \(20260607\), degrees \(n\in\{20,50,100\}\), and \(M=5000\) independent samples for each degree.  Table~\ref{tab:simulation_sharpness} reports the mean, median, and \(90\%\)-quantile of the sharpness ratios.  Values closer to \(1\) indicate sharper certificates.  In all \(15000\) polynomial samples, the deterministic inequalities
\begin{equation}
        \rho_{\rm true}\leq B_C,
        \qquad
        \rho_{\rm true}\leq B_{\rm opt},
        \qquad
        \rho_{\rm true}\leq B_R,
        \qquad
        r_A\leq\lambda_{\rm true}
\end{equation}
were satisfied up to numerical precision.

\begin{table}[t]
\caption{Monte Carlo sharpness statistics for monic Gaussian polynomials.  Each row uses \(M=5000\) independent samples.  The outer ratios are \(B_C/\rho_{\rm true}\), \(B_{\rm opt}/\rho_{\rm true}\), and \(B_R/\rho_{\rm true}\); the inner ratio is \(\lambda_{\rm true}/r_A\).  Values closer to \(1\) indicate sharper certificates.}
\label{tab:simulation_sharpness}
\centering
\begin{tabular}{c l c c c}
\toprule
Degree & Statistic & Mean & Median & 90\% quantile\\
\midrule
20 & \(B_C/\rho_{\rm true}\) & 2.181 & 2.178 & 2.742\\
 & \(B_{\rm opt}/\rho_{\rm true}\) & 1.547 & 1.526 & 1.827\\
 & \(B_R/\rho_{\rm true}\) & 1.242 & 1.215 & 1.455\\
 & \(\lambda_{\rm true}/r_A\) & 2.693 & 2.200 & 4.055\\
 & \((B_R-r_A)/(\rho_{\rm true}-\lambda_{\rm true})\) & 1.856 & 1.716 & 2.655\\
\midrule
50 & \(B_C/\rho_{\rm true}\) & 2.413 & 2.428 & 3.022\\
 & \(B_{\rm opt}/\rho_{\rm true}\) & 1.542 & 1.526 & 1.821\\
 & \(B_R/\rho_{\rm true}\) & 1.237 & 1.213 & 1.446\\
 & \(\lambda_{\rm true}/r_A\) & 3.048 & 2.496 & 4.880\\
 & \((B_R-r_A)/(\rho_{\rm true}-\lambda_{\rm true})\) & 1.887 & 1.762 & 2.725\\
\midrule
100 & \(B_C/\rho_{\rm true}\) & 2.596 & 2.627 & 3.207\\
 & \(B_{\rm opt}/\rho_{\rm true}\) & 1.551 & 1.527 & 1.838\\
 & \(B_R/\rho_{\rm true}\) & 1.243 & 1.213 & 1.464\\
 & \(\lambda_{\rm true}/r_A\) & 3.287 & 2.677 & 5.088\\
 & \((B_R-r_A)/(\rho_{\rm true}-\lambda_{\rm true})\) & 1.928 & 1.790 & 2.779\\
\bottomrule
\end{tabular}
\end{table}

Figure~\ref{fig:outer_radius_ecdf} displays the empirical distribution functions of the true outer radius and the three outer certificates for \(n=50\).  The true radius concentrates close to the unit circle, as expected for Gaussian random polynomials.  The Cauchy radius is visibly conservative.  The optimized Cauchy envelope shifts the certificate distribution substantially toward the true radius, and the Rouch\'e certificate gives the tightest distribution among the tested outer certificates.

\begin{figure}[t]
\centering
\includegraphics[width=0.88\linewidth]{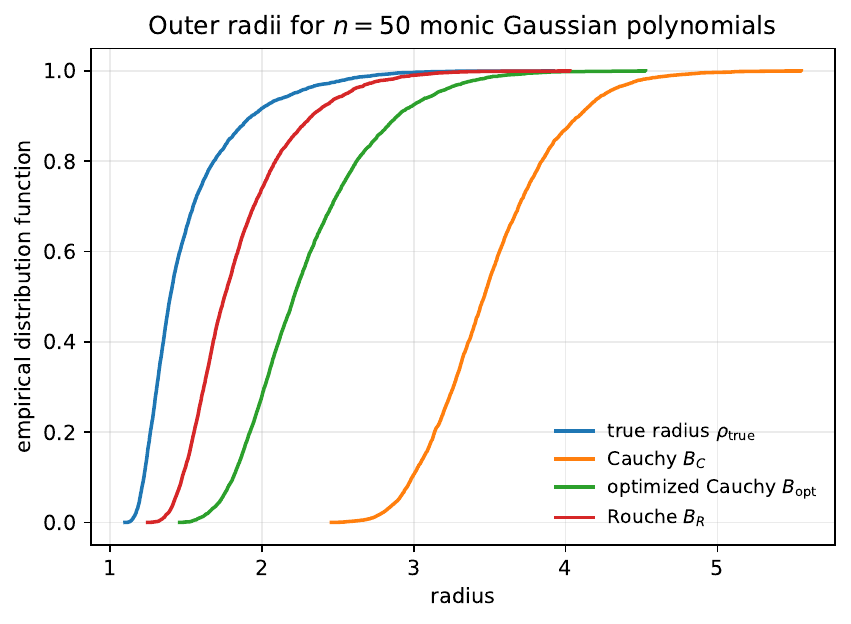}
\caption{Empirical distribution functions of the true outer radius and the three outer certificates for \(n=50\).}
\label{fig:outer_radius_ecdf}
\end{figure}

Figure~\ref{fig:outer_ratio_by_degree} summarizes the median and \(90\%\)-quantile of the outer sharpness ratios as the degree increases.  The classical Cauchy ratio grows mildly with \(n\), while the optimized Cauchy and Rouch\'e ratios remain comparatively stable over this degree range.  The Rouch\'e median is approximately \(1.21\) for all three tested degrees, which means that the Rouch\'e radius exceeds the observed outer zero radius by roughly \(21\%\) at the median.

\begin{figure}[t]
\centering
\includegraphics[width=0.88\linewidth]{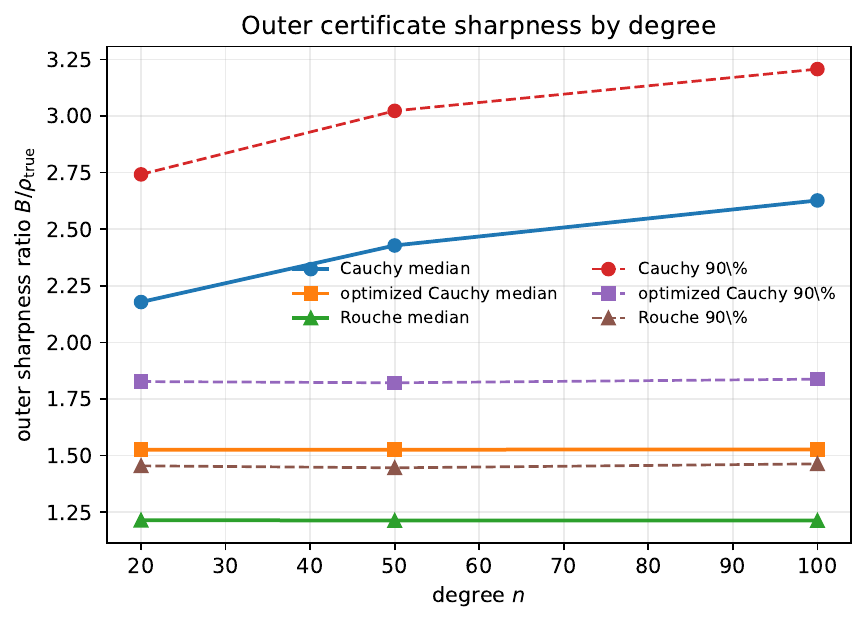}
\caption{Median and \(90\%\)-quantile sharpness ratios for the three outer certificates.}
\label{fig:outer_ratio_by_degree}
\end{figure}

Figure~\ref{fig:inner_ratio_ecdf} shows the empirical distribution of the inner annular sharpness ratio \(\lambda_{\rm true}/r_A\) for \(n=50\).  The median value is \(2.496\), so the reversed-Cauchy lower radius is conservative by a factor of about \(2.5\) at the median in this ensemble.  This is consistent with the structure of the reversed certificate, since it is controlled by ratios to the constant coefficient \(a_0\), and small values of \(|a_0|\) widen the certified annulus.

\begin{figure}[t]
\centering
\includegraphics[width=0.88\linewidth]{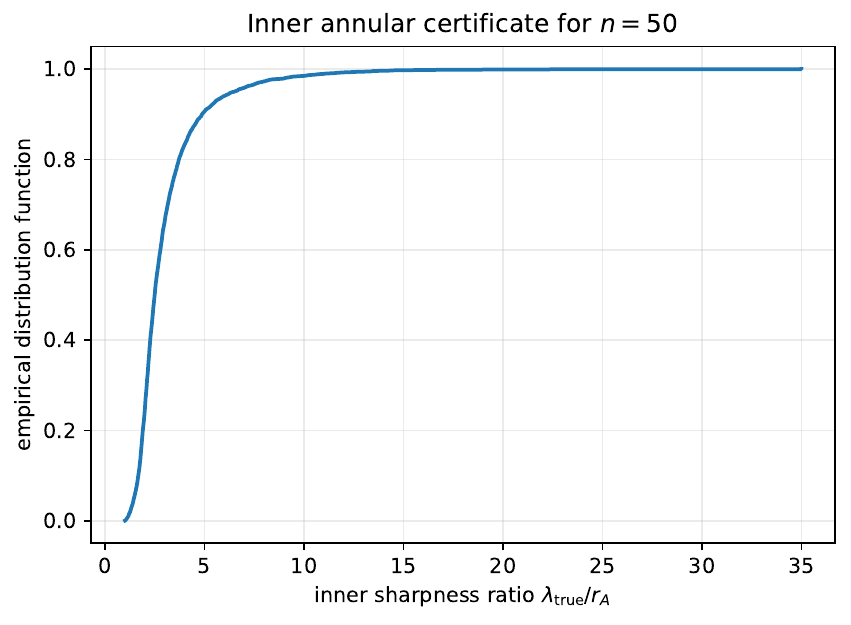}
\caption{Empirical distribution of the inner annular sharpness ratio \(\lambda_{\rm true}/r_A\) for \(n=50\).}
\label{fig:inner_ratio_ecdf}
\end{figure}

Figure~\ref{fig:annular_width_scatter} compares the true annular width \(\rho_{\rm true}-\lambda_{\rm true}\) with the certified width \(B_R-r_A\) for a representative subset of the \(n=50\) samples.  The points lie above the diagonal up to numerical precision, as required by the certificate inequalities.  The median certified-to-true width ratio is \(1.762\) for \(n=50\), which indicates that the two-sided certificate remains informative while retaining rigorous containment.

\begin{figure}[t]
\centering
\includegraphics[width=0.82\linewidth]{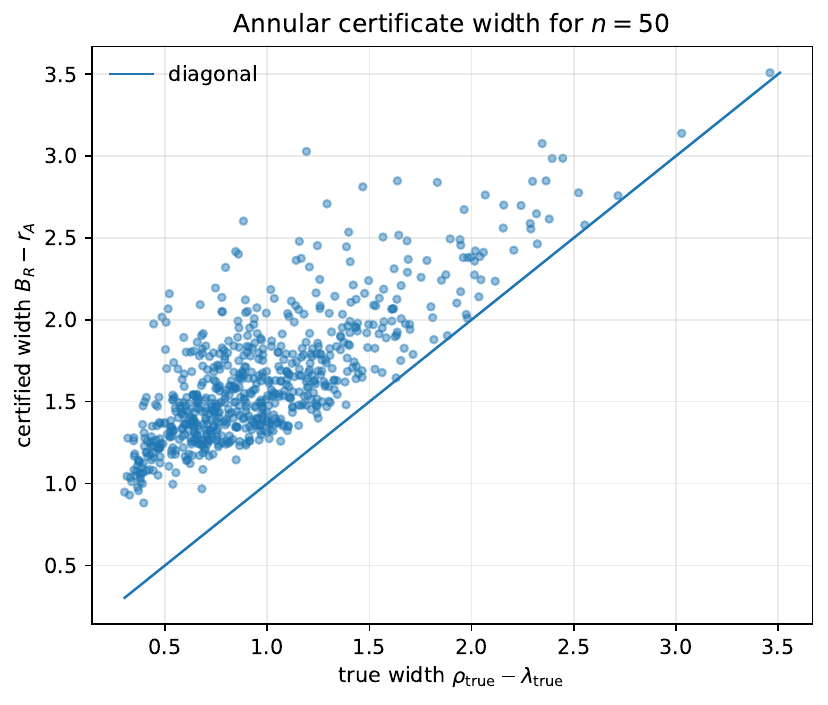}
\caption{True annular width versus certified annular width for a representative subset of \(n=50\) samples.}
\label{fig:annular_width_scatter}
\end{figure}

Figure~\ref{fig:dependent_ratio_quantile} gives a numerical check of the main dependence-aware theorem.  The \(90\%\)-quantile of the dependent Gaussian ratio maximum is computed from the exact integral \(I_n\), and it is compared with the independent-Cauchy quantile and the asymptotic scale \(b_n/a_p\).  The dependent curve follows the logarithmic extreme-value scale, while the independent-Cauchy curve follows the much larger linear scale.  This supports the conclusion that the common denominator in the Gaussian coefficient ratios is a decisive part of the finite-degree certificate.

\begin{figure}[t]
\centering
\includegraphics[width=0.88\linewidth]{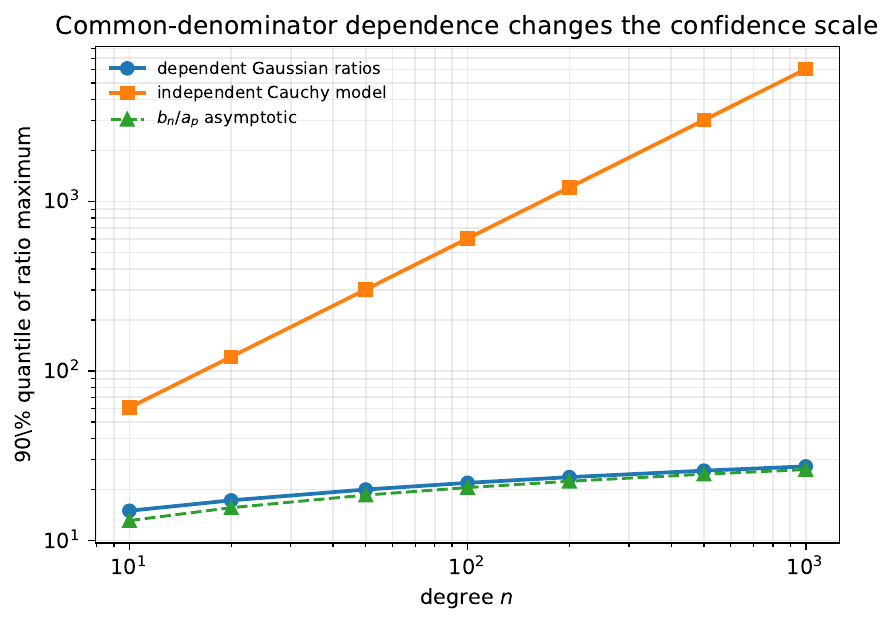}
\caption{Dependence-aware Gaussian ratio quantile compared with the independent-Cauchy model and the asymptotic scale \(b_n/a_p\) at confidence \(p=0.9\).}
\label{fig:dependent_ratio_quantile}
\end{figure}

The numerical study has a deliberately limited scope.  It uses real monic Gaussian coefficients, degrees up to \(100\), and standard companion-polynomial root computation for the diagnostic radii.  The observed ranking of certificates should therefore be read as finite-degree evidence for this ensemble.  Heavy-tailed coefficient laws, complex coefficients, sparse coefficient laws, and very high-degree polynomials may change the empirical sharpness ranking, while the deterministic certificate inequalities themselves remain valid under their stated hypotheses.

\FloatBarrier

\section{Conclusion}

This paper develops finite-degree probabilistic zero certificates for random polynomials.  The framework begins with a deterministic zero-location theorem and treats its bound as a random certificate after the coefficients are sampled.  This gives computable lower or upper estimates for zero-containment probabilities without using the joint law of the zeros.

The manuscript is a single-author continuation of the author's earlier joint work with Mir \cite{SheikhMir2024}.  The earlier paper introduced Cauchy-type probabilistic zero bounds for certain random polynomials.  The present paper keeps that origin explicit and extends it in several directions: a general certificate-transfer theorem, monic coefficient-law formulas, sub-Weibull confidence radii, a scaled Cauchy envelope, annular certificates, and Rouch\'e--Chernoff estimates.

The main new theorem concerns general Gaussian polynomials with random leading coefficient.  Cauchy's bound produces ratios \(X_k/X_n\).  Although each ratio has the standard Cauchy distribution, the ratios are dependent through the common denominator.  The exact dependence-aware certificate probability is
\begin{equation}
        I_n(r)=\sqrt{\frac{2}{\pi}}
        \int_0^\infty e^{-x^2/2}
        \left[
        \erf\left(\frac{rx}{\sqrt2}\right)
        \right]^n\dd x.
\end{equation}
The inverse confidence scale satisfies
\begin{equation}
        I_n^{-1}(p)\sim \frac{b_n}{a_p}
        \asymp \sqrt{\log n},
\end{equation}
whereas independent Cauchy maxima have linear scale in \(n\).  This dependence-aware correction is the central mathematical gain of the certificate formulation.

The simulations add finite-degree evidence for the certificate hierarchy in the monic Gaussian ensemble.  For \(n=20,50,100\) and \(5000\) samples per degree, the Rouch\'e certificate gives the smallest median outer sharpness ratio, around \(1.21\), followed by the optimized Cauchy envelope, around \(1.53\), and then the classical Cauchy certificate.  The annular certificate gives a rigorous two-sided localization region, with a median certified-to-true width ratio between \(1.72\) and \(1.79\) over the tested degrees.  These computations indicate that the certificate-transfer framework is both analytically tractable and numerically testable at finite degree.

Future refinements may replace Cauchy's theorem by sharper deterministic localization theorems, optimize certificate families for special coefficient distributions, and extend the numerical comparison to heavy-tailed, complex, sparse, and structured random polynomials.  The same method also applies to companion matrices, recurrence stability, and randomized polynomial filters, since each of these problems converts zero localization into a finite-sample spectral certificate.

\end{document}